\documentclass[12pt]{amsproc}
\newtheorem{theorem}{\sc Theorem}[section]
\newtheorem{lemma}[theorem]{\sc Lemma}
\newtheorem{proposition}[theorem]{\sc Proposition}

\newtheorem{hypothesis}[theorem]{\sc Hypothesis}

\begin{document}
\title[Finite Conjugacy Classes]{Groups with Boundedly Finite \\ Conjugacy Classes of Commutators}
\author{Gl\'aucia Dierings}
\address{ Department of Mathematics, University of Brasilia,
Brasilia-DF, 70910-900 Brazil}
\email{glauciadierings@yahoo.com.br}
\author{Pavel Shumyatsky }
\address{ Department of Mathematics, University of Brasilia,
Brasilia-DF, 70910-900 Brazil}
\email{pavel@unb.br}
\thanks{This research was supported by FAPDF and CNPq-Brazil}
\begin{abstract} In 1954 B. H. Neumann discovered that if $G$ is a group in which all conjugacy classes are finite with bounded size, then the derived group $G'$ is finite. Later (in 1957) Wiegold found an explicit bound for the order of $G'$. We study groups in which the conjugacy classes containing commutators are finite with bounded size. We obtain the following results.
 
Let $G$ be a group and $n$ a positive integer. 

If $|x^G|\leq n$ for any commutator $x\in G$, then the second derived group $G''$ is finite with $n$-bounded order.

If $|x^{G'}|\leq n$ for any commutator $x\in G$, then the order of  $\gamma_3(G')$ is finite and $n$-bounded. 
\end{abstract}

\maketitle

\section{Introduction} Given a group $G$ and an element $x\in G$, we write $x^G$ for the conjugacy class containing $x$. Of course, if the number of elements in $x^G$ is finite, we have $|x^G|=[G:C_G(x)]$. A group is said to be a BFC-group if its conjugacy classes are finite and of bounded size. One of B. H. Neumann's discoveries was that in a BFC-group the derived group $G'$ is finite \cite{bhn}. It follows that if $|x^G|\leq n$ for each $x\in G$, then the order of $G'$ is bounded by a number depending only on $n$. A first explicit bound for the order of $G'$ was found by J. Wiegold \cite{wie}, and the best known was obtained in \cite{gumaroti} (see also \cite{neuvoe} and \cite{sesha}).

In the present article we deal with groups $G$ such that $|x^G|\leq n$ whenever $x$ is a commutator, that is, $x=[x_1,x_2]$ for suitable $x_1,x_2\in G$. Here and throughout the article we write $[x_1,x_2]$ for $x_1^{-1}x_2^{-1}x_1x_2$. As usual, we denote by $G'$ the derived group of $G$ and by $G''$ the derived group of $G'$ (the second derived group of $G$).

\begin{theorem}\label{main} Let $n$ be a positive integer and $G$ a group in which $|x^G|\leq n$ for any commutator $x$. Then $|G''|$ is finite and $n$-bounded.
\end{theorem}

Further, we consider groups $G$ in which $|x^{G'}|\leq n$ whenever $x$ is a commutator.

\begin{theorem}\label{outro} Let $n$ be a positive integer and $G$ a group in which $|x^{G'}|\leq n$ for any commutator $x$. Then $|\gamma_3(G')|$ is finite and $n$-bounded.
\end{theorem}

Here $\gamma_3(G')$ denotes the third term of the lower central series of $G'$. We do not know whether under hypothesis of Theorem \ref{outro} the second derived group $G''$ must necessarily be finite. Note that under hypothesis of Theorem \ref{main} $\gamma_3(G)$ can be infinite. This can be shown using any example of an infinite torsion-free metabelian group whose commutator quotient is finite (see for instance \cite{gusi}).

We make no attempts to obtain good bounds for $|G''|$ in Theorem \ref{main} and $|\gamma_3(G')|$ in Theorem \ref{outro}. The proofs given here yield bounds $n^{54n^{14}}$ and $n^{12n^{10}}$, respectively. The bounds however do not look realistic at all.
\section{Proofs} 

Let $G$ be a group generated by a set $X$ such that $X=X^{-1}$. Given an element $g\in G$, we write $l_X(g)$ for the minimal number $l$ with the property that $g$ can be written as a product of $l$ elements of $X$. Clearly, $l_X(g)=0$ if and only if $g=1$. We call $l_X(g)$ the length of $g$ with respect to $X$.

\begin{lemma}\label{passma} Let $H$ be a group generated by a set $X=X^{-1}$ and let $K$ be a subgroup of finite index $m$ in $H$. Then each coset $Kb$ contains an element $g$ such that $l_X(g)\leq m-1$.
\end{lemma}
\begin{proof} If $b\in K$, the result is obvious. Therefore we assume that $b\not\in K$. Choose $g\in Kb$ in such a way that $s=l_X(g)$ is as small as possible and suppose that $s\geq m$. Write $g=x_1\cdots x_s$ with $x_i\in X$ and set $y_j=x_1\cdots x_j$ for $j=1,\dots,s$. Since $s$ is the minimum of lengths of elements in $Kb$, it follows that none of the elements $y_1,\dots,y_s$ lies in $K$. Thus, these $s$ elements belong to the union of at most $m-1$ right cosets of $K$ and we conclude that $Ky_i=Ky_j$ for some $1\leq i<j\leq s$. It is now easy to see that the element $h=y_ix_{j+1}\dots x_s$ belongs to $Kb$ while $l_X(h)<l_X(g)$. This is a contradiction with the choice of $g$.
\end{proof}

In the sequel the above lemma will be used in the situation where $H$ is the derived group of a group $G$ and $X$ is the set of commutators in $G$. Therefore we will write $l(g)$ to denote the smallest number such that the element $g\in G'$ can be written as a product of as many commutators. Recall that if $H$ is a group and $a\in H$, the subgroup $[H,a]$ is generated by all commutators of the form $[h,a]$, where $h\in H$. It is well-known that $[H,a]$ is always normal in $H$. Recall that in any group $G$ the following ``standard commutator identities" hold.
\begin{enumerate}
\item $[x,y]^{-1}=[y,x]$;
\item $[xy,z]=[x,z]^y[y,z]$;
\item $[x,yz]=[x,z][x,y]^z$.
\end{enumerate}
In what follows the above identities will be used without explicit references. 

We will now fix some notation and hypothesis.

\begin{hypothesis}\label{01} Let $G$ be a group and $K$ a subgroup containing $H=G'$. Let $X$ denote the set of commutators in $G$ and suppose that $C_K(x)$ has finite index at most $n$ in $K$ for each $x\in X$. Let $m$ be the maximum of indices of $C_H(x)$ in $H$, where $x\in X$.
Suppose further that $a\in X$ and $C_H(a)$ has index precisely $m$ in $H$. Choose $b_1,\dots,b_m\in H$ such that $l(b_i)\leq m-1$ and $a^H=\{a^{b_i};\ i=1,\dots,m\}$. (The existence of such elements is guaranteed by Lemma \ref{passma}.) Set $U=C_K(\langle b_1,\dots,b_m\rangle)$.
\end{hypothesis}

\begin{lemma}\label{02} Assume Hypothesis \ref{01}. Then for any $x\in X$ the subgroup $[H,x]$ has finite $m$-bounded order.
\end{lemma}
\begin{proof} Choose $x\in X$. Since $C_H(x)$ has index at most $m$ in $H$, by Lemma \ref{passma} we can choose elements $y_1,\dots,y_m$ such that $l(y_i)\leq m-1$ and $[H,x]$ is generated by the commutators $[y_i,x]$. For each $i=1,\dots,m$ write $y_i=y_{i1}\dots y_{i(m-1)}$, where $y_{ij}\in X$. The standard commutator identities show that $[y_i,x]$ can be written as a product of conjugates in $H$ of the commutators $[y_{ij},x]$. Let $h_1,\dots,h_s$ be the conjugates in $H$ of elements from the set $\{x,y_{ij};\ 1\leq i,j\leq m\}$. Since $C_H(h)$ has finite index at most $m$ in $H$ for each $h\in X$, it follows that $s$ is $m$-bounded. Let $T=\langle h_1,\dots,h_s\rangle$. It is clear that $[H,x]\leq T'$ and so it is sufficient to show that $T'$ has finite $m$-bounded order. Observe that $C_H(h_i)$ has finite index at most $m$ in $H$ for each $i=1,\dots,s$. It follows that the center $Z(T)$ has index at most $m^s$ in $T$. Thus, Schur's theorem \cite[10.1.4]{Rob} tells us that $T'$ has finite $m$-bounded order, as required.
\end{proof}

Note that the subgroup $U$ has finite $n$-bounded index in $K$. This follows from the facts that $l(b_i)\leq m-1$ and $C_K(x)$ has index at most $n$ in $K$ for each $x\in X$.

The next lemma is somewhat analogous with Lemma 4.5 of Wiegold \cite{wie}.
\begin{lemma}\label{03} Assume Hypothesis \ref{01}. Suppose that $u\in U$ and $ua\in X$. Then $[H,u]\leq[H,a]$.
\end{lemma}
\begin{proof} Since $u\in U$, it follows that $(ua)^{b_i}=ua^{b_i}$ for each $i=1,\dots,m$. Therefore the elements $ua^{b_i}$ form the conjugacy class $(ua)^H$. For an arbitrary element $g\in H$ there exists $h\in\{b_1,\dots,b_m\}$ such that $(ua)^{g}=ua^{h}$ and so $u^ga^g=ua^h$. Therefore $[u,g]=a^ha^{-g}\in[H,a]$. The lemma follows.
\end{proof}

\begin{proposition}\label{04} Assume Hypothesis \ref{01} and write $a=[d,e]$ for suitable $d,e\in G$. There exists a subgroup $U_1\leq U$ with the following properties.
\begin{enumerate}
\item The index of $U_1$ in $K$ is $n$-bounded;
\item $[H,U_1']\leq[H,a]^{d^{-1}}$;
\item $[H,[U_1,d]]\leq[H,a]$.
\end{enumerate}
\end{proposition}
\begin{proof} Set $$U_1=U\cap U^{d^{-1}}\cap U^{d^{-1}e^{-1}}.$$ Since the index of $U$ in $K$ is $n$-bounded, we conclude that the index of $U_1$ in $K$ is $n$-bounded as well. Choose arbitrarily elements $h_1,h_2\in U_1$. Write $$[h_1d,eh_2]=[h_1,h_2]^d[d,h_2][h_1,e]^{dh_2}[d,e]^{h_2}$$
and so 
$$[h_1d,eh_2]^{h_2^{-1}}=[h_1,h_2]^{dh_2^{-1}}[d,h_2]^{h_2^{-1}}[h_1,e]^d[d,e].$$ Denote the product $[h_1,h_2]^{dh_2^{-1}}[d,h_2]^{h_2^{-1}}[h_1,e]^d$ by $u$. Thus, the right hand side of the above equality is $ua$ while, obviously, on  the left hand side we have a commutator. Let us check that $u\in U$. We see that $[h_1,h_2]^{dh_2^{-1}}\in U_1^{dh_2^{-1}}\leq U$ because $U_1^d\leq U$. By the same reason, $[d,h_2]^{h_2^{-1}}\in U$. Finally, $[h_1,e]^d\in U_1^dU_1^{ed}\leq U$ so indeed $u\in U$. By Lemma \ref{03}, $[H,u]\leq[H,a]$. This holds for any choice of $h_1,h_2\in U_1$. In particular, taking $h_1=1$ we see that $[H,[d,h_2]^{h_2^{-1}}]\leq[H,a]$ while taking $h_2=1$ we conclude that $[H,[h_1,e]^d] \leq [H,a]$. It now follows that $[H,[h_1,h_2]^{dh_2^{-1}}]\leq[H,a]$. Since $[H,a]$ is normal in $H$, we have $[H,[h_1,h_2]]\leq[H,a]^{d^{-1}}$ and so $[H,U_1']\leq[H,a]^{d^{-1}}$, which proves that $U_1$ has property 2. Examine again the inclusion $[H,[d,h_2]^{h_2^{-1}}]\leq[H,a]$. Since $[H,a]$ is normal in $H$, it follows that $[H,[U_1,d]]\leq[H,a]$. Therefore $U_1$ has property 3 as well. The proof is now complete.
\end{proof}

We are ready to prove our main results.

\begin{proof}[Proof of Theorem \ref{main}] Recall that $G$ is a group in which $|x^G|\leq n$ for any commutator $x$. We need to show that $|G''|$ is finite and $n$-bounded.

We denote by $X$ the set of commutators in $G$ and set $H=G'$. Let $m$ be the maximum of indices of $C_H(x)$ in $H$, where $x\in X$. Of course, $m\leq n$. Choose $a\in X$ such that $C_H(a)$ has index precisely $m$ in $H$. Choose $b_1,\dots,b_m\in H$ such that $l(b_i)\leq m-1$ and $a^H=\{a^{b_i};\ i=1,\dots,m\}$. Set $U=C_G(\langle b_1,\dots,b_m\rangle)$. Note that the index of $U$ in $G$ is $n$-bounded. Applying Proposition \ref{04} with $K=G$ we find a subgroup $U_1$, of $n$-bounded index, such that $[H,U_1']\leq\langle[H,a]^G\rangle$. Since the index of $U_1$ in $G$ is $n$-bounded, we can find $n$-boundedly many commutators $c_1,\dots,c_s\in X$ such that $H=\langle c_1,\dots,c_s,H\cap U_1\rangle$. Let $T$ be the normal closure in $G$ of the product of the subgroups $[H,a]$ and $[H,c_i]$ for $i=1,\dots,s$. By Lemma \ref{02} each of these subgroups has $n$-bounded order. Our hypothesis is that each of them has at most $n$ conjugates. Thus, $T$ is a product of $n$-boundedly many finite subgroups, normalizing each other and having $n$-bounded order. We conclude that $T$ has finite $n$-bounded order. Therefore it is sufficient to show that the second derived group of the quotient $G/T$ has finite $n$-bounded order. So we pass to the quotient $G/T$. To avoid complicated notation the images of $G$, $H$ and $X$ will be denoted by the same symbols. We observe that the derived group of $HU_1$ is contained in $Z(H)$. This follows from the facts that $HU_1$ is generated by $c_1,\dots,c_s$ and $U_1$ and modulo $T$ we have $c_1,\dots,c_s\in Z(H)$ and $U_1'\leq Z(H)$.

Let $\mathcal X$ denote the family of subgroups $S\leq G$ with the following properties.
\begin{enumerate}
\item $H\leq S$;
\item $S'\leq Z(H)$;
\item $S$ has finite index in $G$.
\end{enumerate}
We already know that $\mathcal X$ is non-empty since it contains $HU_1$. Choose $J\in\mathcal X$ of minimal possible index $j$ in $G$. Since the index of $U_1$ in $G$ is $n$-bounded, the index $j$ is $n$-bounded, too. We will now use induction on $j$. If $j=1$, then $J=G$ and $H\leq Z(H)$. So $G''=1$ and we have nothing to prove. Thus, we assume that $j\geq2$.

Again, we take a commutator $a_0\in X$ such that $C_H(a_0)$ has maximal possible index in $H$ and write $a_0=[d,e]$ for suitable $d,e\in G$. If both $d$ and $e$ belong to $J$, we conclude (since $J'\leq Z(H)$) that $H$ is abelian and $G''=1$. Thus, assume that at least one of them, say $d$, is not in $J$. We will use Proposition \ref{04} with $K=G$. It follows that there is a subgroup $V$ of $n$-bounded index in $G$ such that $[H,[V,d]]\leq[H,a_0]$. Replacing if necessary $V$ by $V\cap J$, without loss of generality we can assume that $V\leq J$. Let $L=J\langle d\rangle$. Note that $L'=J'[J,d]$. Let $1=g_1,\dots,g_t$ be a full system of representatives of the right cosets of $V$ in $J$. Then $[J,d]$ is generated by $[V,d]^{g_1},\dots,[V,d]^{g_t}$ and $[g_1,d],\dots,[g_t,d]$. This is straightforward from the fact that $[vg,d]=[v,d]^g[g,d]$ for any $g,v\in G$. Next, for each $i=1,\dots,t$ set $x_i=[g_i,d]$. Let $R$ be the normal closure in $G$ of the product of the subgroups $[H,{a_0}]^{g_i}$ and $[H,x_i]$ for $i=1,\dots,t$. By Lemma \ref{02} each of these subgroups has $n$-bounded order. Our hypothesis is that each of them has at most $n$ conjugates. Thus, $R$ is a product of $n$-boundedly many finite subgroups, normalizing each other and having $n$-bounded order. We conclude that $R$ has finite $n$-bounded order. We see that $[H,L']\leq R$. Since $d\not\in J$, the index of $L$ in $G$ is strictly smaller than $j$. Therefore, by induction on $j$, the second derived group of $G/R$ is finite with bounded order. Taking into account that also $R$ is finite with bounded order, we deduce that $G''$ is finite with bounded order. The proof is now complete.
\end{proof}

\begin{proof}[Proof of Theorem \ref{outro}.] Recall that $G$ is a group in which $|x^{G'}|\leq n$ for any commutator $x$. We need to prove that $\gamma_3(G')$ is finite with $n$-bounded order. As before, we write $X$ for the set of commutators in $G$ and $H$ for the derived group. Choose a commutator $a\in X$ such that $C_H(a)$ has maximal possible index in $H$. We will use Proposition \ref{04} with $K=H$. It follows that $H$ contains a subgroup $U_1$ of finite $n$-bounded index such that $[H,U_1']\leq[H,a]^{d^{-1}}$ for some $d\in G$. Write $b_0=a^{d^{-1}}$ and so $[H,U_1']\leq[H,b_0]$. Since the index of $U_1$ in $H$ is $n$-bounded, we can find $n$-boundedly many commutators $b_1,\dots,b_s\in X$ such that $H=\langle b_1,\dots,b_s,U_1\rangle$. Let $T$ be the product of the subgroups $[H,b_i]$  for $i=0,1,\dots,s$. By Lemma \ref{02} each of these subgroups has $n$-bounded order. All of them are normal in $H$ and so $T$ is normal in $H$ and has finite $n$-bounded order. The center of $H/T$ contains the images of $U_1'$ and $b_1,\dots,b_s$. It follows that the quotient of $H/T$ over its center is abelian. Therefore $\gamma_3(H)\leq T$, which completes the proof.
\end{proof}

\end{document}